\documentclass[11pt]{amsart}

\usepackage{amsmath}
\usepackage{amssymb}
\usepackage{amscd}
\usepackage{color}

\begin{document}

\title{Strong $J$-cleanness of formal matrix rings}

\author{Orhan Gurgun}
\address{Orhan Gurgun, Department of Mathematics, Ankara University,  Turkey}
\email{orhangurgun@gmail.com}

\author{Sait Halicioglu}
\address{Sait Hal\i c\i oglu,  Department of Mathematics, Ankara University, Turkey}
\email{halici@ankara.edu.tr}

\author{Abdullah Harmanci}
\address{Abdullah Harmanci, Department of Maths, Hacettepe University,  Turkey}
\email{harmanci@hacettepe.edu.tr}

\date{\empty}
\date{}
\newtheorem{thm}{Theorem}[section]
\newtheorem{lem}[thm]{Lemma}
\newtheorem{prop}[thm]{Proposition}
\newtheorem{cor}[thm]{Corollary}
\newtheorem{exs}[thm]{Examples}
\newtheorem{defn}[thm]{Definition}
\newtheorem{nota}{Notation}
\newtheorem{rem}[thm]{Remark}
\newtheorem{ex}[thm]{Example}

\maketitle

\begin{abstract}An element $a$ of a ring $R$ is called \emph{strongly $J$-clean} provided that there
exists an idempotent $e\in R$ such that $a-e\in J(R)$ and $ae=ea$.
A ring $R$ is \emph{strongly $J$-clean} in case every element in
$R$ is strongly $J$-clean. In this paper, we investigate strong
$J$-cleanness of $M_2(R;s)$ for a local ring $R$ and $s\in R$. We
determine the conditions under which elements of $M_2(R;s)$ are
strongly $J$-clean.
 \\[+2mm]
{\bf Keywords:} Strongly $J$-clean, strongly clean, strongly nil
clean, formal matrix ring, local ring.
\thanks{ \\{\bf 2010 Mathematics Subject Classification:} 16S50,
16S70, 16U99}
\end{abstract}

\section{Introduction}

Throughout this paper all rings are associative with identity
unless otherwise stated. We write $R[[x]]$, $U(R)$, $C(R)$,
$Nil(R)$ and $J(R)$ for the power series ring over a ring $R$, the
set of all invertible elements, the set of all central elements,
the set of all nilpotent elements and the Jacobson radical of $R$,
respectively.

Following \cite{N}, we define an element $a$ of a ring $R$ to be
\emph{clean} if there is an idempotent $e\in R$ such that $a-e$ is
a unit of $R$. A \emph{clean ring} is defined to be one in which
every element is clean. In \cite{N2}, an element $a$ of a ring $R$
is defined to be \emph{strongly clean} if there is an idempotent
$e\in R$, which commutes with $a$, such that $a-e$ is invertible
in $R$. Analogously, a \emph{strongly clean ring} is one in which
every element is strongly clean. Chen \cite{C1} defines strongly
$J$-clean rings. An element $a\in R$ is called \emph{strongly
$J$-clean} provided that there exists an idempotent $e\in R$ such
that $a-e\in J(R)$ and $ae=ea$. A ring $R$ is \emph{strongly
$J$-clean} in case every element in $R$ is strongly $J$-clean.
Many variations of the notions of clean and strongly clean have
been studied by a variety of authors and it is still a very
popular subject.

The question of when a matrix ring $M_n(R)$ is strongly clean has
been discussed by several authors (see \cite{C, C1, C2, D, TZ,
YZ}). Recently, Tang and Zhou \cite{TZ, TZ2} completely
characterized the local rings $R$ for which the formal matrix ring
$M_2(R;s)$ is strongly clean. In this article, we are motivated to
investigate the strong $J$-cleanness of the formal matrix rings
over a local ring.

Given a ring $R$ and an element $s\in C(R)$, the $4$-tuple $\left[%
\begin{array}{cc}
  R & R \\
  R & R \\
\end{array}%
\right]$ becomes a ring with addition defined componentwise and
with multiplication defined by

 $$\begin{array}{c}
\left[
\begin{array}{cc}
a&b\\
c&d
\end{array}
\right]\left[
\begin{array}{cc}
a'&b'\\
c'&d'
\end{array}
\right]=\left[
\begin{array}{cc}
aa'+s^2bc'&ab'+bd'\\
ca'+dc'&s^2cb'+dd'
\end{array}
\right].
\end{array}$$

\noindent This ring is denoted by $M_2(R;s)$ in \cite{TZ2}. When
$s = 1$, $M_2(R;s)$ is just the matrix ring $M_2(R)$, but
generally $M_2(R,s)$ can be significantly different from $M_2(R)$
(see \cite{TZ2}). Some properties of $M_2(R;s)$ were studied in \cite{TZ2}. For $s\in C(R)$, the $4$-tuple $\left[%
\begin{array}{cc}
  R & R \\
  R & R \\
\end{array}%
\right]$ becomes a ring with addition defined componentwise and
with multiplication defined by

 $$\begin{array}{c}
\left[
\begin{array}{cc}
a&b\\
c&d
\end{array}
\right]\left[
\begin{array}{cc}
a'&b'\\
c'&d'
\end{array}
\right]=\left[
\begin{array}{cc}
aa'+sbc'&ab'+bd'\\
ca'+dc'&scb'+dd'
\end{array}
\right].
\end{array}$$

\noindent This ring is denoted by $K_s(R)$ by Krylov in \cite{K},
and is discussed in \cite{K, KT, TZ}. In our notation,
$M_2(R;s)=K_{s^2}(R)$. By \cite[Proposition 4]{TZ2}, for a local
ring $R$ and an element $s\in C(R)$, $M_2(R;s)\cong M_2(R)$ if and
only if $s\in U(R)$.

In what follows,  the ring of integers modulo $n$ is denoted by
$\mathbb{Z}_n$, and we write $M_n(R)$ (resp. $T_n(R)$) for the
rings of all (resp., all upper triangular) $n\times n$ matrices
over the ring $R$. For $s\in C(R)$, $\big(s:J(R)\big)=\{x\in
R~|~sx\in J(R)\}$ is denoted by $J_s(R)$. Note that
$J_s(R)=J_{s^2}(R)$ for $s\in C(R)$ (see \cite[Lemma 10]{TZ2}).
For elements $a,b$ in a ring $R$, we use the notation $a\sim b$ to
mean that $a$ \textit{is similar to} $b$, that is, $b=u^{-1}au$
for some $u\in U(R)$. For elements $a,b$ in a ring $R$, we say
that \textit{$a$ is equivalent to $b$} if there exist $u,v\in
U(R)$ such that $b=uav$.

\section{Strongly $J$-Clean Elements}

In this section, we consider the question of when $M_2(R;s)$ is
strongly $J$-clean for a local ring $R$. We start with some useful
results.

\begin{thm}\label{teoor1}\cite[Theorem 11]{TZ2} Let $R$ be a ring
with $s\in C(R)$. Then $J\big(M_2(R;s)\big)=\left[%
\begin{array}{cc}
  J(R) & J_s(R) \\
  J_s(R) & J(R) \\
\end{array}%
\right]$.
\end{thm}

\begin{lem} \label{lem2} Let $R$ be a local ring with $s\in C(R)\cap J(R)$. Then $J\big(M_2(R;s)\big)=\left[%
\begin{array}{cc}
  J(R) & R \\
  R & J(R) \\
\end{array}%
\right]$ and, moreover, $\left[%
\begin{array}{cc}
 a & x\\
  y & b \\
\end{array}%
\right]\in U\big(M_2(R;s)\big)$ if and only if $a,b\in U(R)$.
\end{lem}

\begin{proof} It is straightforward.
\end{proof}

\begin{ex}\label{ex1} \rm{ Let $R$ be a ring.
\begin{itemize}
\item[(1)] Every element in $J(R)$ is strongly $J$-clean.
\item[(2)] $u\in U(R)$ and $u$ is strongly $J$-clean if and only if $u-1\in J(R)$.
\item[(3)] $a\in R$ is strongly $J$-clean if and only if $1-a\in
R$ is strongly $J$-clean.
\end{itemize}}
\end{ex}

\begin{lem} \label{lem4} Let $R$ be a ring with $s\in C(R)$. Then $A\in
U\big(M_2(R;s)\big)$ and $A$ is strongly $J$-clean if and only if
$A-I_2\in J\big(M_2(R;s)\big)$.
\end{lem}

\begin{proof} Assume that $A\in
U\big(M_2(R;s)\big)$ and $A$ is strongly $J$-clean. By
Example~\ref{ex1}, $A-I_2\in J\big(M_2(R;s)\big)$. Conversely,
suppose that $A-I_2\in J\big(M_2(R;s)\big)$. Then $A$ is strongly
$J$-clean. Since $A\in 1+J\big(M_2(R;s)\big)\subseteq
U\big(M_2(R;s)\big)$, $A\in U\big(M_2(R;s)\big)$. The proof is
completed.
\end{proof}

\begin{lem} \label{lem6} Let $R$ be a ring with $s\in C(R)$ and let $P\in U\big(M_2(R;s)\big)$. Then $A\in M_2(R;s)$ is
strongly $J$-clean if and only if $PAP^{-1}\in M_2(R;s)$ is
strongly $J$-clean.
\end{lem}

\begin{proof} Assume that $A\in M_2(R;s)$ is
strongly $J$-clean. Then there exists $E^2=E\in M_2(R;s)$ such
that $A-E=W\in J\big(M_2(R;s)\big)$ and $EW=WE$. Let $F=PEP^{-1}$
and $V=PWP^{-1}$. Then $F^2=F$ and $V\in J\big(M_2(R;s)\big)$.
Since $A-E=W\in J\big(M_2(R;s)\big)$, we have $PAP^{-1}-F=V\in
J\big(M_2(R;s)\big)$. Further, since $EW=WE$,
$FV=PEP^{-1}PWP^{-1}=PEWP^{-1}=PWEP^{-1}=PWP^{-1}PEP^{-1}=VF$.
Hence $PAP^{-1}\in M_2(R;s)$ is strongly $J$-clean. Clearly, if
$PAP^{-1}\in M_2(R;s)$ is strongly $J$-clean for some $P\in
U\big(M_2(R;s)\big)$, then $A\in M_2(R;s)$ is strongly $J$-clean.
\end{proof}

\begin{lem}\label{temel1} Let $R$ be a local ring with $s\in C(R)\cap J(R)$.
Then $A^n\in J\big(M_2(R;s)\big)$ for some $n\in \Bbb N$ if and
only if $A\in J\big(M_2(R;s)\big)$.
\end{lem}

\begin{proof} Assume that $A^n\in J\big(M_2(R;s)\big)$ for some $n\in \Bbb N$ and let $A=\left[%
\begin{array}{cc}
 a & b\\
 c & d \\
\end{array}%
\right]$ where $a, b, c, d\in R$. This gives $A^n=\left[%
\begin{array}{cc}
 a^n+s^2\ast & \ast\\
 \ast & d^n+s^2\ast \\
\end{array}%
\right]\in J\big(M_2(R;s)\big)$. Then $a^n+s^2\ast\in J(R)$ and
$d^n+s^2\ast\in J(R)$ by Lemma~\ref{lem2}. Since $R$ is a local
ring and $s\in J(R)$, we have $a,d\in J(R)$. Hence, by
Lemma~\ref{lem2}, $A\in J\big(M_2(R;s)\big)$.
\end{proof}

Recall that an element $a$ of a ring $R$ is \textit{strongly nil
clean} in case there is an idempotent $e\in R$ such that $ae=ea$
and $a-e\in Nil(R)$. A ring $R$ is \textit{strongly nil clean} in
case every element in R is strongly nil clean (see \cite{D}). Note
that this result shows that if $A$ is strongly nil clean in
$M_2(R;s)$, then $A$ is strongly $J$-clean for any $s\in C(R)\cap
J(R)$ because $A\in Nil\big(M_2(R;s)\big)$ implies that $A^n=0\in
J\big(M_2(R;s)\big)$ for some $n\in \Bbb N$.

Recall that an element $\left[%
\begin{array}{cc}
 a & 0\\
  0 & b \\
\end{array}%
\right]$ is called \textit{a diagonal matrix of} $M_2(R;s)$.

\begin{lem} \label{idempotent} Let $R$ be a local ring with $s\in C(R)$ and let $E$ be a non-trivial idempotent of $M_2(R;s)$.
Then we have the following.

\begin{itemize}
\item[(1)] If $s\in U(R)$, then $E\sim \left[%
\begin{array}{cc}
 1 & 0\\
  0 & 0 \\
\end{array}%
\right]$.
\item[(2)] If $s\in J(R)$, then either $E\sim \left[%
\begin{array}{cc}
 1 & 0\\
  0 & 0 \\
\end{array}%
\right]$ or $E\sim \left[%
\begin{array}{cc}
 0 & 0\\
  0 & 1 \\
\end{array}%
\right]$.
\end{itemize}
\end{lem}

\begin{proof} Let $E=\left[%
\begin{array}{cc}
 a & b\\
  c & d \\
\end{array}%
\right]$ where $a,b,c,d\in R$. Since $E^2=E$, we have
$$a^2+s^2bc=a, ~s^2cb+d^2=d, ~ab+bd=b, ~ca+dc=c.$$

\noindent If $a,d\in J(R)$, then $b,c\in J(R)$ and so $E\in
J\big(M_2(R;s)\big)$. Hence $E=0$, a
contradiction. Since $R$ is local, we have $a\in U(R)$ or $d\in U(R)$. If $d\in U(R)$, then $\left[%
\begin{array}{cc}
 1 & -bd^{-1}\\
  0 & 1 \\
\end{array}%
\right]\left[%
\begin{array}{cc}
 a & b\\
  c & d \\
\end{array}%
\right]\left[%
\begin{array}{cc}
 1 & 0\\
  -d^{-1}c & d^{-1} \\
\end{array}%
\right]=\left[%
\begin{array}{cc}
 a-s^2bd^{-1}c & 0\\
  0 & 1 \\
\end{array}%
\right]$ and so $E$ is equivalent to a diagonal matrix. If $a\in
U(R)$, then we similarly show that $E$ is equivalent to a diagonal
matrix. According to \cite[Corollary 26]{TZ2},
there exists $P\in U\big(M_2(R;s)\big)$ such that $PEP^{-1}=\left[%
\begin{array}{cc}
 f & 0\\
  0 & g \\
\end{array}%
\right]$, where $f,g\in R$ are idempotents. Since $R$ is local and
$E\neq 0$, we see that either $f=1$ and $g=0$ or $f=0$ and $g=1$.
If $f=1$ and $g=0$, then $E\sim \left[%
\begin{array}{cc}
 1 & 0\\
  0 & 0 \\
\end{array}%
\right]$. Let $f=0$ and $g=1$. If $s\in U(R)$, then $\left[%
\begin{array}{cc}
 0 & 1\\
  1 & 0 \\
\end{array}%
\right]\left[%
\begin{array}{cc}
 0 & 0\\
  0 & 1 \\
\end{array}%
\right]\left[%
\begin{array}{cc}
 0 & s^{-2}\\
  s^{-2} & 0 \\
\end{array}%
\right]=\left[%
\begin{array}{cc}
 1 & 0\\
  0 & 0 \\
\end{array}%
\right]$ and so $E\sim \left[%
\begin{array}{cc}
 1 & 0\\
  0 & 0 \\
\end{array}%
\right]$. If $s\in J(R)$ and $E\sim \left[%
\begin{array}{cc}
 1 & 0\\
  0 & 0 \\
\end{array}%
\right]$, then it is easy to check that $s\in U(R)$, a
contradiction. Hence, if $s\in J(R)$, then either $E\sim \left[%
\begin{array}{cc}
 1 & 0\\
  0 & 0 \\
\end{array}%
\right]$ or $E\sim \left[%
\begin{array}{cc}
 0 & 0\\
  0 & 1 \\
\end{array}%
\right]$.
\end{proof}

\begin{lem} \label{lem7} Let $R$ be a local ring and $s\in C(R)$.
Then $A\in M_2(R;s)$ is strongly $J$-clean if and only if one of
the following holds:

\begin{itemize}
\item[(1)] $A\in J\big(M_2(R;s)\big)$, or
\item[(2)] $I_2-A\in J\big(M_2(R;s)\big)$, or
\item[(3)] $A\sim \left[%
\begin{array}{cc}
 v & 0\\
  0 & w \\
\end{array}%
\right]$, where $v\in 1+J(R)$, $w\in J(R)$ and $s\in U(R)$, or
\item[(4)] either $A\sim \left[%
\begin{array}{cc}
 v & 0\\
  0 & w \\
\end{array}%
\right]$ or $A\sim \left[%
\begin{array}{cc}
 w & 0\\
  0 & v \\
\end{array}%
\right]$, where $v\in 1+J(R)$, $w\in J(R)$ and $s\in J(R)$.
\end{itemize}
\end{lem}

\begin{proof} ``$\Leftarrow:$" If either $A\in J\big(M_2(R;s)\big)$ or $I_2-A\in
J\big(M_2(R;s)\big)$, then $A$ is strongly $J$-clean. If $A\sim \left[%
\begin{array}{cc}
 v & 0\\
  0 & w \\
\end{array}%
\right]$ where $v\in 1+J(R)$, $w\in J(R)$ and $s\in U(R)$, then $\left[%
\begin{array}{cc}
 v & 0\\
  0 & w \\
\end{array}%
\right]=\left[%
\begin{array}{cc}
 1 & 0\\
  0 & 0 \\
\end{array}%
\right]+\left[%
\begin{array}{cc}
 v-1 & 0\\
  0 & w \\
\end{array}%
\right]$ is strongly $J$-clean in $M_2(R;s)$. By Lemma~\ref{lem6},
$A$ is strongly $J$-clean. Similarly, if either $A\sim \left[%
\begin{array}{cc}
 v & 0\\
  0 & w \\
\end{array}%
\right]$ or $A\sim \left[%
\begin{array}{cc}
 w & 0\\
  0 & v \\
\end{array}%
\right]$, where $v\in 1+J(R)$, $w\in J(R)$ and $s\in J(R)$, then
$A$ is strongly $J$-clean.

``$\Rightarrow:$" Assume that $A$ is strongly $J$-clean and both
$A, I_2-A\notin J\big(M_2(R;s)\big)$. Then there exists $E^2=E\in
M_2(R;s)$ such that $A-E=W\in J\big(M_2(R;s)\big)$ and $EW=WE$.
Since both $A, I_2-A\notin J\big(M_2(R;s)\big)$, $E$ is a
non-trivial
idempotent of $M_2(R;s)$. If $s\in U(R)$, then $E\sim \left[%
\begin{array}{cc}
 1 & 0\\
  0 & 0 \\
\end{array}%
\right]$ by Lemma~\ref{idempotent}. That is, there exists $P\in
U\big(M_2(R;s)\big)$ such that $PEP^{-1}=\left[%
\begin{array}{cc}
 1 & 0\\
  0 & 0 \\
\end{array}%
\right]$. According to Lemma~\ref{lem6},
$PAP^{-1}=PEP^{-1}+PWP^{-1}$ is strongly $J$-clean in $M_2(R;s)$.
Let $V=[v_{ij}]=PWP^{-1}$. Since $V\left[%
\begin{array}{cc}
 1 & 0\\
  0 & 0 \\
\end{array}%
\right]=\left[%
\begin{array}{cc}
 1 & 0\\
  0 & 0 \\
\end{array}%
\right]V$, we have $v_{12}=v_{21}=0$ and so $v_{11}, v_{22}\in
J(R)$.
Hence $A\sim \left[%
\begin{array}{cc}
 1+v_{11} & 0\\
  0 &  v_{22}\\
\end{array}%
\right]$. Similarly, if $s\in J(R)$, then either $A\sim \left[%
\begin{array}{cc}
 1+v_{11} & 0\\
  0 &  v_{22}\\
\end{array}%
\right]$ or $A\sim \left[%
\begin{array}{cc}
 v_{11} & 0\\
  0 &  1+v_{22}\\
\end{array}%
\right]$.
\end{proof}

\begin{lem} \label{lem1.1} Let $R$ be a local ring and $s\in C(R)$.
Then $A\in M_2(R;s)$ is strongly nil clean if and only if one of
the following holds:

\begin{itemize}
\item[(1)] $A\in Nil\big(M_2(R;s)\big)$,
\item[(2)] $I_2-A\in Nil\big(M_2(R;s)\big)$,
\item[(3)] $A\sim \left[%
\begin{array}{cc}
 v & 0\\
  0 & w \\
\end{array}%
\right]$, where $v\in 1+Nil(R)$, $w\in Nil(R)$ and $s\in U(R)$,
\item[(4)] either $A\sim \left[%
\begin{array}{cc}
 v & 0\\
  0 & w \\
\end{array}%
\right]$ or $A\sim \left[%
\begin{array}{cc}
 w & 0\\
  0 & v \\
\end{array}%
\right]$, where $v\in 1+Nil(R)$, $w\in Nil(R)$ and $s\in J(R)$.
\end{itemize}
\end{lem}

\begin{proof} It is similar to the proof of Lemma~\ref{lem7}.
\end{proof}

We now extend \cite[Corollary 5.4]{C1} to $M_2(R;s)$.

\begin{cor}\label{cor2} Let $R$ be a commutative local ring with $s\in U(R)$. Then $M_2(R;s)$ is not strongly $J$-clean.
\end{cor}

\begin{proof} Let $A=\left[%
\begin{array}{cc}
 1 & 1\\
 1 & 0 \\
\end{array}%
\right]\in M_2(R;s)$. Clearly, $A\notin J\big(M_2(R;s)\big)$ and
$I_2-A\notin J\big(M_2(R;s)\big)$. It is easy to check that $A\in
U\big(M_2(R;s)\big)$ and $I_2-A\in U\big(M_2(R;s)\big)$. Assume
that $A$ is strongly $J$-clean in $M_2(R;s)$. By Lemma~\ref{lem4},
$A-I_2\in J\big(M_2(R;s)\big)$, a contradiction. Hence $M_2(R;s)$
is not strongly $J$-clean.
\end{proof}

\begin{lem} \label{lem8} Let $R$ be a local ring and $s\in C(R)$.
Then $A\in M_2(R;s)$ is strongly clean if and only if one of the
following holds:

\begin{itemize}
\item[(1)] $A\in U\big(M_2(R;s)\big)$,
\item[(2)] $I_2-A\in U\big(M_2(R;s)\big)$,
\item[(3)] $A\sim \left[%
\begin{array}{cc}
 w & 0\\
  0 & v \\
\end{array}%
\right]$, where $v\in 1+J(R)$, $w\in J(R)$ and $s\in U(R)$,
\item[(4)] either $A\sim \left[%
\begin{array}{cc}
 v & 0\\
  0 & w \\
\end{array}%
\right]$ or $A\sim \left[%
\begin{array}{cc}
 w & 0\\
  0 & v \\
\end{array}%
\right]$, where $v\in 1+J(R)$, $w\in J(R)$ and $s\in J(R)$.
\end{itemize}
\end{lem}

\begin{proof} ``$\Leftarrow:$" If either $A\in U\big(M_2(R;s)\big)$ or $I_2-A\in
U\big(M_2(R;s)\big)$, then $A$ is strongly clean. If $A\sim \left[%
\begin{array}{cc}
 v & 0\\
  0 & w \\
\end{array}%
\right]$ where $v\in 1+J(R)$, $w\in J(R)$ and $s\in U(R)$, then $\left[%
\begin{array}{cc}
 v & 0\\
  0 & w \\
\end{array}%
\right]=\left[%
\begin{array}{cc}
 0 & 0\\
  0 & 1 \\
\end{array}%
\right]+\left[%
\begin{array}{cc}
 v & 0\\
  0 & w-1 \\
\end{array}%
\right]$ is strongly clean in $M_2(R;s)$. So
$A$ is strongly clean. Similarly, if either $A\sim \left[%
\begin{array}{cc}
 v & 0\\
  0 & w \\
\end{array}%
\right]$ or $A\sim \left[%
\begin{array}{cc}
 w & 0\\
  0 & v \\
\end{array}%
\right]$, where $v\in 1+J(R)$, $w\in J(R)$ and $s\in J(R)$, then
$A$ is strongly clean.

``$\Rightarrow:$" Assume that $A$ is strongly clean and both $A,
I_2-A\notin U\big(M_2(R;s)\big)$. As in the proof of
Lemma~\ref{lem7}, if $s\in U(R)$, then we show that $A\sim B=\left[%
\begin{array}{cc}
 1+v_{11} & 0\\
  0 &  v_{22}\\
\end{array}%
\right]$ where $v_{11}, v_{22}\in U(R)$. Note that $A\in
U\big(M_2(R;s)\big)$ if and only if $PAP^{-1}\in
U\big(M_2(R;s)\big)$. This gives that $B\notin
U\big(M_2(R;s)\big)$ and $I_2-B\notin U\big(M_2(R;s)\big)$. Since
$R$ is local, we have $v=v_{22}\in 1+J(R)$
and $w=1+v_{11}\in J(R)$. Similarly, if $s\in J(R)$, then either $A\sim \left[%
\begin{array}{cc}
 v & 0\\
  0 & w \\
\end{array}%
\right]$ or $A\sim \left[%
\begin{array}{cc}
 w & 0\\
  0 & v \\
\end{array}%
\right]$, where $v\in 1+J(R)$, $w\in J(R)$.
\end{proof}

Hence the following corollary is immediate.

\begin{cor}\label{faydali} Let $R$ be a local ring and $s\in C(R)$.
Then $A\in M_2(R;s)$ is strongly clean if and only if one of the
following holds:

\begin{itemize}
\item[(1)] $A\in U\big(M_2(R;s)\big)$,
\item[(2)] $I_2-A\in U\big(M_2(R;s)\big)$,
\item[(3)] $A\sim \left[%
\begin{array}{cc}
 a & 0\\
  0 & b \\
\end{array}%
\right]$, where $a,b\in R$.
\end{itemize}
\end{cor}

We now extend \cite[Theorem 5.5]{C1} to $M_2(R;s)$.

\begin{thm}\label{teo1} Let $R$ be a local ring with $s\in C(R)$. Then $A\in M_2(R;s)$ is strongly clean if and only if one of the following holds:
\begin{itemize}
\item[(1)] $A\in U\big(M_2(R;s)\big)$,
\item[(2)] $I_2-A\in U\big(M_2(R;s)\big)$,
\item[(3)] $A\in M_2(R;s)$ is strongly $J$-clean.
\end{itemize}
\end{thm}

\begin{proof} Assume that $A\in M_2(R;s)$ is strongly clean. We may assume that $A\notin
U\big(M_2(R;s)\big)$ and $I_2-A\notin U\big(M_2(R;s)\big)$. If $s\in U(R)$, $A\sim \left[%
\begin{array}{cc}
 w & 0\\
  0 & v\\
\end{array}%
\right]$, where $v\in 1+J(R)$ and $w\in J(R)$, by Lemma~\ref{lem8}. Since $\left[%
\begin{array}{cc}
 w & 0\\
  0 & v \\
\end{array}%
\right]=\left[%
\begin{array}{cc}
 0 & 0\\
  0 & 1 \\
\end{array}%
\right]+\left[%
\begin{array}{cc}
 w & 0\\
  0 & v-1 \\
\end{array}%
\right]$ is strongly $J$-clean, $A$ is strongly $J$-clean. If $s\in J(R)$, either $A\sim\left[%
\begin{array}{cc}
 v & 0\\
  0 & w \\
\end{array}%
\right]$ or $A\sim \left[%
\begin{array}{cc}
 w & 0\\
  0 & v \\
\end{array}%
\right]$, where $v\in 1+J(R)$ and $w\in J(R)$, by Lemma~\ref{lem8}. Since $\left[%
\begin{array}{cc}
 v & 0\\
  0 & w \\
\end{array}%
\right]=\left[%
\begin{array}{cc}
 1 & 0\\
  0 & 0 \\
\end{array}%
\right]+\left[%
\begin{array}{cc}
 v-1 & 0\\
  0 & w \\
\end{array}%
\right]$ is strongly $J$-clean, $A$ is strongly $J$-clean. Similarly, if $A\sim\left[%
\begin{array}{cc}
 w & 0\\
  0 & v \\
\end{array}%
\right]$, then $A$ is strongly $J$-clean. So holds (3).
Conversely, let $A\in M_2(R;s)$. If either (1) or (2), then $A$ is
strongly clean by Lemma~\ref{lem8}. Hence we may assume that $A$
is strongly $J$-clean. According to \cite[Proposition 2.1]{C1},
$A$ is strongly clean in $M_2(R;s)$.
\end{proof}

\begin{prop}\label{prop4} Let $R$ be a local ring and $s\in C(R)$. Then the following are equivalent.
\begin{enumerate}
    \item For every $A\in M_2(R;s)$ with $A, I_2-A\notin
    U\big(M_2(R;s)\big)$ is strongly nil clean.
    \item For every $A\in M_2(R;s)$ with $A, I_2-A\notin
    U\big(M_2(R;s)\big)$ is strongly $J$-clean and $J(R)$ is nil.
\end{enumerate}
\end{prop}

\begin{proof} $(1)\Rightarrow(2)$ Let $x\in J(R)$ and let $A=\left[%
\begin{array}{cc}
 x & 0\\
 0 & 1 \\
\end{array}%
\right]\in M_2(R;s)$. Since $A\notin Nil\big(M_2(R;s)\big)$ and
$I_2-A\notin Nil\big(M_2(R;s)\big)$, we have $A\sim \left[%
\begin{array}{cc}
 w & 0\\
  0 & v \\
\end{array}%
\right]$, where $v\in 1+Nil(R)$, $w\in Nil(R)$ by
Lemma~\ref{lem1.1}. Therefore, $x\in Nil(R)$ and so $J(R)$ is nil.
By Lemma~\ref{temel1}, for every $A\in M_2(R;s)$ with $A,
I_2-A\notin U\big(M_2(R;s)\big)$ if $A$ is strongly nil clean,
then it is strongly $J$-clean.

$(2)\Rightarrow(1)$ Since $J(R)$ is nil, it is obvious.
\end{proof}

\begin{lem}\label{temel} Let $R$ be a local ring with $s\in J(R)\cap
C(R)$. Assume that $A\in M_2(R;s)$ such that neither $A$ nor
$I_2-A$
is a unit. Then $A$ is similar to $\left[%
\begin{array}{cc}
  u & 1\\
  v & w \\
\end{array}%
\right]$ or $\left[%
\begin{array}{cc}
  w & 1\\
  v & u\\
\end{array}%
\right]$, where $u\in 1+J(R)$, $v\in U(R)$, $w\in J(R)$.
\end{lem}

\begin{proof} It is similar to the proof of \cite[Lemma 5]{TZ}.
\end{proof}

\begin{thm}\label{temelteo1} Let $R$ be a local ring with $s\in J(R)\cap C(R)$.
Then $A\in M_2(R;s)$ is strongly $J$-clean if and only if one of
the following holds:
\begin{itemize}
\item[(1)] $A\in J\big(M_2(R;s)\big)$,
\item[(2)] $I_2-A\in J\big(M_2(R;s)\big)$,
\item[(3)] $A$ is similar to $\left[%
\begin{array}{cc}
  u & 1\\
  v & w \\
\end{array}%
\right]$, where $u\in 1+J(R)$, $v\in U(R)$, $w\in J(R)$, and the
equation $t^2 -(vuv^{-1} +w)t +(vuv^{-1}w -s^2v)$ has a right root
in $1+J(R)$ and $t^2 -(u+w)t +(wu-s^2v)$ has a right root in
$J(R)$.
\item[(4)] $A$ is similar to $\left[%
\begin{array}{cc}
  w & 1\\
  v & u\\
\end{array}%
\right]$, where $u\in 1+J(R)$, $v\in U(R)$, $w\in J(R)$, and the
equation $t^2 -(u+vwv^{-1})t +(vwv^{-1}u -s^2v)$ has a right root
in $J(R)$ and $t^2 -(u+w)t +(uw-s^2v)$ has a right root in
$1+J(R)$.
\end{itemize}
\end{thm}

\begin{proof} Assume that $A$ is strongly $J$-clean in $M_2(R;s)$ and $A, I_2-A\notin
J\big(M_2(R;s)\big)$. Then $A$ is strongly clean. If $A\in
U\big(M_2(R;s)\big)$ or $I_2-A\in U\big(M_2(R;s)\big)$, then
$I_2-A\in J\big(M_2(R;s)\big)$ or $A\in J\big(M_2(R;s)\big)$,
respectively, by Example~\ref{ex1}(3) and Lemma~\ref{lem4}, a
contradiction. Hence we may assume that $A\notin
U\big(M_2(R;s)\big)$ and $I_2-A\notin
U\big(M_2(R;s)\big)$. By Lemma~\ref{temel}, $A$ is similar to either $\left[%
\begin{array}{cc}
  u & 1\\
  v & w \\
\end{array}%
\right]$ or $\left[%
\begin{array}{cc}
  w & 1\\
  v & u\\
\end{array}%
\right]$, where $u\in 1+J(R)$, $v\in U(R)$, $w\in J(R)$. If $A\sim \left[%
\begin{array}{cc}
  u & 1\\
  v & w \\
\end{array}%
\right]$, then the equation $t^2 -(vuv^{-1} +w)t +(vuv^{-1}w
-s^2v)$ has a right root in $1+J(R)$ and $t^2 -(u+w)t
+(wu-s^2v)$ has a right root in $J(R)$ by \cite[Lemma 7]{TZ}. If $A\sim \left[%
\begin{array}{cc}
  w & 1\\
  v & u \\
\end{array}%
\right]$, then we similarly show that the equation  $t^2
-(u+vwv^{-1})t +(vwv^{-1}u -s^2v)$ has a right root in $J(R)$ and
$t^2 -(u+w)t +(uw-s^2v)$ has a right root in $1+J(R)$.

We now prove the converse. If $(1)$ or $(2)$ holds, then $A$ is
strongly $J$-clean. Suppose that $(3)$ holds. It suffices to show
that $B=\left[%
\begin{array}{cc}
  u & 1\\
  v & w \\
\end{array}%
\right]$ is a strongly $J$-clean element of $M_2(R;s)$. By
\cite[Lemma 7]{TZ}, $B$ is strongly clean. Since $B\notin
U\big(M_2(R;s)\big)$ and $I_2-B\notin U\big(M_2(R;s)\big)$, $B$ is
strongly $J$-clean by Theorem~\ref{teo1}. If (4) holds, we
similarly show that $B$ is strongly $J$-clean.
\end{proof}

For a ring $R$, let $a\in R$, $l_a : R\rightarrow R$ and $r_a :
R\rightarrow R$ denote, respectively, the abelian group
endomorphisms given by $l_a(r) = ar$ and $r_a(r) = ra$ for all
$r\in R$. Thus, for $a$, $b\in R$, $l_a$, $r_b$ is an abelian
group endomorphism such that $(l_a - r_b)(r) = ar - rb$ for any
$r\in R$. A local ring R is called {\it weakly bleached} \cite{YZ}
if, for any $a\in J(R)$ and any $b\in 1+J(R)$, the abelian group
endomorphisms $l_b - r_a$ and $l_a - r_b$ of $R$ are both
surjective.

\begin{lem} \label{chen} Let $R$ be a local ring with $s=\sum\limits_{i=1}^{\infty}s_ix^i\in C\big(R[[x]]\big)\cap
J\big(R[[x]]\big)$, and let $A(x)\in M_2\big(R[[x]];s\big)$. If
$A(0)$ is similar to $\left[%
\begin{array}{cc}
  u & 1\\
  v & w \\
\end{array}%
\right]$ or $\left[%
\begin{array}{cc}
  w & 1\\
  v & u\\
\end{array}%
\right]$, where $u\in 1+J(R)$, $v\in U(R)$, $w\in J(R)$, then
$A(x)$ is similar to $\left[%
\begin{array}{cc}
  u(x) & 1\\
  v(x) & w(x) \\
\end{array}%
\right]$ or $\left[%
\begin{array}{cc}
  w(x) & 1\\
  v(x) & u(x)\\
\end{array}%
\right]$, where $u(x)\in 1+J\big(R[[x]]\big)$, $v(x)\in
U\big(R[[x]]\big)$, $w(x)\in J\big(R[[x]]\big)$ and $u(0)=u$,
$v(0)=v$, $w(0)=w$.
\end{lem}

\begin{proof} Assume that $A(0)$ is similar to $\left[%
\begin{array}{cc}
  u & 1\\
  v & w \\
\end{array}%
\right]$, where $u\in 1+J(R)$, $v\in U(R)$, $w\in J(R)$. That is,
there exists a $P\in U\big(M_2(R;s)\big)$ such that $P^{-1}A(0)P=\left[%
\begin{array}{cc}
  u & 1\\
  v & w \\
\end{array}%
\right]$. Then $$P^{-1}A(x)P=\left[%
\begin{array}{cc}
  p(x) & 1+q(x)x\\
  r(x) & t(x) \\
\end{array}%
\right],$$

\noindent where $q(x)\in R[[x]]$, $p(x)\in 1+J\big(R[[x]]\big)$,
$r(x)\in U\big(R[[x]]\big)$ and $t(x)\in J\big(R[[x]]\big)$. Hence

$
\begin{array}{lll}
A(x)&\sim & \left[%
\begin{array}{cc}
  \big(1+q(x)x\big)^{-1} & 0\\
  0 & 1 \\
\end{array}%
\right]P^{-1}A(x)P\left[%
\begin{array}{cc}
  1+q(x)x & 0\\
  0 & 1 \\
\end{array}%
\right]\\
 &=&\left[%
\begin{array}{cc}
  \big(1+q(x)x\big)^{-1}p(x)\big(1+q(x)x\big) & 1\\
  r(x)\big(1+q(x)x\big) & t(x)\\
\end{array}%
\right]\\
\end{array}
$

\noindent Let $u(x)=\big(1+q(x)x\big)^{-1}p(x)\big(1+q(x)x\big)$,
$v(x)=r(x)\big(1+q(x)x\big)$ and $w(x)= t(x)$. It is easy to
verify that $u(x)\in 1+J\big(R[[x]]\big)$, $v(x)\in
U\big(R[[x]]\big)$, $w(x)\in J\big(R[[x]]\big)$ and $u(0)=u$,
$v(0)=v$, $w(0)=w$. If $A(0)$ is similar to $\left[%
\begin{array}{cc}
  w & 1\\
  v & u\\
\end{array}%
\right]$, where $u\in 1+J(R)$, $v\in U(R)$, $w\in J(R)$, then
we analogously show that $A(x)$ is similar to $\left[%
\begin{array}{cc}
  w(x) & 1\\
  v(x) & u(x)\\
\end{array}%
\right]$, where $u(x)\in 1+J\big(R[[x]]\big)$, $v(x)\in
U\big(R[[x]]\big)$, $w(x)\in J\big(R[[x]]\big)$ and $u(0)=u$,
$v(0)=v$, $w(0)=w$.
\end{proof}

\begin{thm}\label{power}Let $R$ be a weakly bleached local ring with
$s=\sum\limits_{i=1}^{\infty}s_ix^i\in C\big(R[[x]]\big)\cap
J\big(R[[x]]\big)$. The following are equivalent.
\begin{enumerate}
\item $A(x)\in M_2\big(R[[x]]; s\big)$ is strongly $J$-clean.
\item $A(0)\in M_2(R;s_0)$ is strongly $J$-clean.
\end{enumerate}
\end{thm}

\begin{proof} $(1)\Rightarrow(2)$ It is clear.

$(2)\Rightarrow(1)$ Obviously, $R[[x]]$ is local. If $A(0)\in
J\big(M_2(R;s)\big)$, then $A(x)$ is strongly $J$-clean. If
$I_2-A(0)\in J\big(M_2(R;s)\big)$, then $A(x)$ is strongly
$J$-clean. Since $A(0)\in M_2(R;s_0)$ is strongly $J$-clean, by
Theorem~\ref{temelteo1}, $A(0)$ is similar to $\left[%
\begin{array}{cc}
  u & 1\\
  v & w \\
\end{array}%
\right]$ or $\left[%
\begin{array}{cc}
  w & 1\\
  v & u\\
\end{array}%
\right]$, where $u\in 1+J(R)$, $v\in U(R)$, $w\in J(R)$. Assume
that $A(0)$ is similar to $\left[%
\begin{array}{cc}
  u & 1\\
  v & w \\
\end{array}%
\right]$. Then, by Lemma~\ref{chen}, $A(x)$ is similar to $\left[%
\begin{array}{cc}
  u(x) & 1\\
  v(x) & w(x) \\
\end{array}%
\right]$, where $u(x)\in 1+J\big(R[[x]]\big)$, $v(x)\in
U\big(R[[x]]\big)$, $w(x)\in J\big(R[[x]]\big)$ and $u(0)=u$,
$v(0)=v$, $w(0)=w$. Further, by Theorem~\ref{temelteo1}, the
equation $t^2 -(vuv^{-1} +w)t +(vuv^{-1}w -s^2v)$ has a right root
$\alpha\in 1+J(R)$ and $t^2 -(u+w)t +(wu-s^2v)$ has a right root
$\beta\in J(R)$. Write $r=\sum\limits_{i=0}^{\infty}r_ix^i\in
R[[x]]$. we have only to show that the equation $t^2
-\big(v(x)u(x)v(x)^{-1} +w(x)\big)t +\big(v(x)u(x)v(x)^{-1}w(x)
-s^2v(x)\big)$ has a right root in $1+J\big(R[[x]]\big)$ and $t^2
-\big(u(x)+w(x)\big)t +\big(w(x)u(x)-s^2v(x)\big)$ has a right
root in $J\big(R[[x]]\big)$. Hence, $r^2 -\big(u(x)+w(x)\big)r
+\big(w(x)u(x)-s^2v(x)\big)=0$ holds if the following hold:

$$\begin{array}{r}
r_0^2-(u+w)r_0+(wu-s_0^2v)=0,\\
(r_0r_1+r_1r_0)-(u+w)r_1-(u_1+w_1)r_0=\ast,\\
(r_0r_2+r_1^2+r_2r_0)-(u+w)r_2-(u_1+w_1)r_1-(u_2+w_2)r_0=\ast,\\
\vdots
\end{array}$$

\noindent Let $r_0=\alpha$. Since $R$ is a weakly bleached local
ring, there exists some $r_1\in R$ such that
$$r_1r_0-(u+w-r_0)r_1=\ast.$$ Further, there exists some $r_2\in R$ such that
$$r_2r_0-(u+w-r_0)r_2=\ast.$$ By
iteration of this process, we get $r_3,r_4,\cdots $. Then
$t^2-\big(u(x)+w(x)\big)t +\big(w(x)u(x)-s^2v(x)\big)$ has a right
root $r\in J\big(R[[x]]\big)$. Let
$v(x)u(x)v(x)^{-1}=\sum\limits_{i=0}^{\infty}\lambda_ix^i$ and
$y=\sum\limits_{i=0}^{\infty}y_ix^i$. Thus, the equation $y^2
-\big(v(x)u(x)v(x)^{-1}+w(x)\big)y +\big(v(x)u(x)v(x)^{-1}w(x)
-s^2v(x)\big)=0$ holds if the following hold:

$$\begin{array}{r}
y_0^2-(vuv^{-1}+w)y_0+(vuv^{-1}w-s_0^2v)=0,\\
(y_0y_1+y_1y_0)-(vuv^{-1}+w)y_1-(\lambda_1+w_1)y_0=\ast,\\
(y_0y_2+y_1^2+y_2y_0)-(vuv^{-1}+w)y_2-(\lambda_1+w_1)y_1-(\lambda_2+w_2)y_0=\ast,\\
\vdots
\end{array}$$

\noindent Let $y_0=\beta$. Since $R$ is a weakly bleached local
ring, there exists some $y_1\in R$ such that
$$y_1y_0-(vuv^{-1}+w-y_0)y_1=\ast.$$ Further, there exists some $y_2\in R$ such that
$$y_2y_0-(vuv^{-1}+w-y_0)y_2=\ast.$$ By
iteration of this process, we get $y_3,y_4,\cdots $. This gives
that the equation $t^2 -\big(v(x)u(x)v(x)^{-1} +w(x)\big)t
+\big(v(x)u(x)v(x)^{-1}w(x) -s^2v(x)\big)$ has a right root $y\in
1+J\big(R[[x]]\big)$. In view of Theorem~\ref{temelteo1}, $A(x)$
is strongly $J$-clean in $M_2\big(R[[x]]; s\big)$.
\end{proof}

\begin{cor}\label{power2} Let $R$ be a weakly bleached local ring with
$s=\sum\limits_{i=1}^{\infty}s_ix^i\in C\big(R[[x]]\big)\cap
J\big(R[[x]]\big)$. The following are equivalent.
\begin{enumerate}
\item $A(x)\in M_2\big(R[[x]]; s\big)$ is strongly clean.
\item $A(0)\in M_2(R;s_0)$ is strongly clean.
\end{enumerate}
\end{cor}

\begin{proof} $(1)\Rightarrow(2)$ is clear.

$(2)\Rightarrow(1)$ By Theorem~\ref{teo1}, we consider the
following cases:

\begin{itemize}
    \item[(i)] $A(0)\in U\big(M_2(R;s)\big)$,
    \item[(ii)] $I_2-A(0)\in U\big(M_2(R;s)\big)$,
    \item[(iii)] $A(0)$ is strongly $J$-clean.
\end{itemize}

\noindent If $A(0)\in U\big(M_2(R;s)\big)$, then $A(x)$ is
strongly clean in $M_2\big(R[[x]];s\big)$. If $I_2-A(0)\in
U\big(M_2(R;s)\big)$, then $A(x)\in M_2\big(R[[x]];s\big)$ is
strongly clean. If $A(0)$ is strongly $J$-clean, then $A(x)$ is
strongly $J$-clean by Theorem~\ref{power}. Hence $A(x)$ is
strongly clean in $M_2\big(R[[x]];s\big)$.
\end{proof}

\section{Formal matrix rings over a commutative local ring}

The aim of this section is to investigate the strong $J$-cleanness
of a single element of $M_2(R;s)$ for a commutative local ring $R$
with $s\in R$. Our results yield the main results of \cite{C1}
when specializing to $s=1$.

If $R$ is a commutative ring with $s\in R$ and $A=\left[%
\begin{array}{cc}
 a & b\\
 c & d \\
\end{array}%
\right]\in M_2(R;s)$, we define $det_s(A)=ad-s^2bc$ and
$tr(A)=a+d$,
and $rA=\left[%
\begin{array}{cc}
 ra & rb\\
 rc & rd\\
\end{array}%
\right]$ for $r\in R$ (see \cite{TZ2}).

\begin{lem}\label{temel2}\cite[Proposition 32]{TZ2} Let $R$ be a commutative ring with $s\in R$ and let $A, B\in M_2(R;s)$. The following hold:
\begin{itemize}
\item[(1)] $det_s(AB)=det_s(A)det_s(B)$.
\item[(2)] $A\in U\big(M_2(R;s)\big)$ if and only if $det_s(A)\in
U(R)$. In this case, if $A=\left[%
\begin{array}{cc}
 a & b\\
 c & d \\
\end{array}%
\right]$, then $A^{-1}= det_s(A)^{-1}\left[%
\begin{array}{cc}
 d & -b\\
 -c & a \\
\end{array}%
\right]$.
\item[(3)] If $A\sim B$, then $det_s(A)=det_s(B)$ and $tr(A)=tr(B)$.
\end{itemize}
\end{lem}

Note that $det_s(I_2-A)=1-tr(A)+det_s(A)$ for any $A\in M_2(R;s)$.

\begin{prop}\label{prop1} Let $R$ be a commutative ring with $s\in R$ and let $A\in M_2(R;s)$
with $A\notin J\big(M_2(R;s)\big)$. If $det_s(A)\in J(R)$ and
$tr(A)\in J(R)$, then $A$ is not strongly $J$-clean in $M_2(R;s)$.
\end{prop}

\begin{proof} Assume that $A\in M_2(R;s)$ is strongly $J$-clean. By
Example~\ref{ex1}, $I_2-A$ is strongly $J$-clean in $M_2(R;s)$. By
the remark above, $det_s(I_2-A)=1-tr(A)+det_s(A)\in U(R)$ and so
$I_2-A\in U\big(M_2(R;s)\big)$ by Lemma~\ref{temel2}(2). According
to Lemma~\ref{lem4}, $(I_2-A)-I_2=-A\in J\big(M_2(R;s)\big)$,
which contradicts the fact that $A\notin J\big(M_2(R;s)\big)$.
\end{proof}



\begin{prop}\label{prop0} Let $R$ be a commutative local ring with $s\in R$ and let $A\in M_2(R;s)$.
Then $det_s(A), tr(A)\in J(R)$ and $A$ is strongly $J$-clean if
and only if $A\in J\big(M_2(R;s)\big)$.
\end{prop}

\begin{proof} ``$\Rightarrow:$ " Let $A=\left[%
\begin{array}{cc}
  a & b \\
  c & d \\
\end{array}%
\right]$ and we will write $\overline{r}=r+R$ and $\overline{R}=R/J(R)$. If $s\in U(R)$, then $J\big(M_2(R;s)\big)=\left[%
\begin{array}{cc}
  J(R)& J(R) \\
  J(R) & J(R) \\
\end{array}%
\right]$ and so we have $J\big(M_2(R;s)\big)=M_2\big(J(R);s\big)$.
Since $det_s(A)=ad-s^2bc\in J(R)$ and $tr(A)=a+d\in J(R)$, the
$s$-characteristic polynomial of $\overline{A}\in
M_2\big(R/J(R);\overline{s}\big)$ is simply $t^2$. By
\cite[Theorem 34]{TZ2}, $\overline{A}^2=\overline{0}$ and so
$A^2\in J\big(M_2(R;s)\big)$. On the other hand, as $A$ is
strongly $J$-clean in $M_2(R;s)$, there exists an $E^2=E\in
M_2(R;s)$ and $V\in J\big(M_2(R;s)\big)$ such that $A=E+V$ and
$EV=VE$. Then $(A-V)^2=A^2-2AV+V^2=E\in J\big(M_2(R;s)\big)$ and
so $E=0\in M_2(R;s)$. Hence $A\in J\big(M_2(R;s)\big)$. If $s\in
J(R)$, then $J\big(M_2(R;s)\big)=\left[%
\begin{array}{cc}
  J(R)& R \\
  R & J(R) \\
\end{array}%
\right]$. This gives that $ad\in J(R)$ and $a+d\in J(R)$. Hence
$a,d\in J(R)$. Therefore $A\in J\big(M_2(R;s)\big)$.

``$\Leftarrow:$ " Let $A=\left[%
\begin{array}{cc}
  a & b \\
  c & d \\
\end{array}%
\right]$. Clearly, $A\in M_2(R;s)$ is strongly $J$-clean. If $s\in
U(R)$, then $a,b,c,d\in J(R)$ and so $det_s(A), tr(A)\in J(R)$. If
$s\in J(R)$, then $a,d\in J(R)$. This implies that $det_s(A),
tr(A)\in J(R)$.
\end{proof}

\begin{rem} \rm{In Proposition~\ref{prop0}, it was proved that $det_s(A), tr(A)\in J(R)$ if
and only if $A\in J\big(M_2(R;s)\big)$ for any $s\in J(R)$.}
\end{rem}

\begin{lem} \label{gerekli} Let $R$ be a commutative local ring with $s\in R$. Then every upper triangular matrix in
$M_2(R;s)$ is strongly clean.
\end{lem}

\begin{proof} Let $A=\left[%
\begin{array}{cc}
 a & b\\
 0 & c \\
\end{array}%
\right]\in M_2(R;s)$ where $a,b,c\in R$. We can assume that
$det_s(A)=ac\in J(R)$ and $tr(A)=a+c\in U(R)$ by
Lemma~\ref{lem8}. This gives $c-a\in U(R)$. Choose $P=\left[%
\begin{array}{cc}
 1 & b(c-a)^{-1}\\
 0 & 1 \\
\end{array}%
\right]$. By Lemma~\ref{temel2}, $P\in U\big(M_2(R;s)\big)$,
and a direct calculation shows that $P^{-1}AP=\left[%
\begin{array}{cc}
 a & 0\\
 0 & c \\
\end{array}%
\right]$. Then $A$ is strongly clean from Corollary~\ref{faydali}.
\end{proof}

\begin{thm}\label{teo3} Let $R$ be a commutative local ring with $s\in R$.
Then $A\in M_2(R;s)$ is strongly $J$-clean if and only if one of
the following holds:
\begin{itemize}
\item[(1)] $A\in J\big(M_2(R;s)\big)$,
\item[(2)] $I_2-A\in J\big(M_2(R;s)\big)$,
\item[(3)] The $s$-characteristic polynomial $t^2-tr(A)t+det_s(A)=0$ has a root in
$J(R)$ and a root in $1+J(R)$.
\end{itemize}
\end{thm}

\begin{proof} ``$\Rightarrow:$" We may assume that $A\notin J\big(M_2(R;s)\big)$ and
$I_2-A\notin J\big(M_2(R;s)\big)$. If $s\in J(R)$, then either $A\sim \left[%
\begin{array}{cc}
 1+v & 0\\
  0 & w \\
\end{array}%
\right]$ or $A\sim \left[%
\begin{array}{cc}
 v & 0\\
  0 & 1+w \\
\end{array}%
\right]$, where $v,w\in J(R)$ by Lemma~\ref{lem7}. Let $A\sim B=\left[%
\begin{array}{cc}
 v & 0\\
  0 & 1+w \\
\end{array}%
\right]$. In view of Lemma~\ref{temel2}, we have
$tr(A)=tr(B)=v+w+1$ and $det_s(A)=det_s(B)=v(1+w)$. Hence, the
$s$-characteristic polynomial $x^2-tr(A)x+det_s(A)=0$ has a root
$v\in J(R)$ and a root
$1+w\in 1+J(R)$. If $A\sim \left[%
\begin{array}{cc}
 1+v & 0\\
  0 & w \\
\end{array}%
\right]$ then, analogously, we show that $(2)$ holds. If $s\in U(R)$, then $A\sim \left[%
\begin{array}{cc}
 1+v & 0\\
  0 & w \\
\end{array}%
\right]$, where $v,w\in J(R)$ by Lemma~\ref{lem7}. Similarly, we
see that $(2)$ holds.

``$\Leftarrow:$" Let $A=\left[%
\begin{array}{cc}
 a & b\\
 c & d \\
\end{array}%
\right]\in M_2(R;s)$. If either $(1)$ or $(2)$ holds, then $A$ is
strongly $J$-clean in $M_2(R;s)$. Suppose that the
$s$-characteristic polynomial $t^2-tr(A)t+det_s(A)=0$ has a root
$y\in J(R)$ and a root $x\in 1+J(R)$. Since
$det_s(A)=ad-s^2bc=xy\in J(R)$ and $tr(A)=a+d=x+y\in 1+J(R)$,
$A\notin U\big(M_2(R;s)\big)$ and $A-I_2\notin
U\big(M_2(R;s)\big)$. So one of $a, d$ must be in $1+J(R)$ and the
other must be in $J(R)$. Without loss of generality we may assume
that $a\in U(R)$. Let $P=\left[%
\begin{array}{cr}
 1 & 0\\
 c(x-d)^{-1} & 1 \\
\end{array}%
\right]$. By Lemma~\ref{temel2}(2),  $P\in U\big(M_2(R;s)\big)$
and easy calculation shows that $P^{-1}AP$ is an upper triangular
matrix in $M_2(R;s)$. In view of Lemma~\ref{gerekli}, $A$ is a
non-trivial strongly clean element of $M_2(R;s)$. By
Theorem~\ref{teo1}, $A\in M_2(R;s)$ is strongly $J$-clean.
\end{proof}

\end{document}